\documentclass[12pt]{amsart}

\usepackage{amsfonts}
\usepackage{amscd}
\usepackage{amssymb}

\allowdisplaybreaks
\usepackage{graphicx}
\usepackage{xcolor}
\usepackage{comment}
\usepackage{array}
\usepackage{tikz-cd}
\usepackage{cite}
\usepackage{enumerate,enumitem,hyperref,cleveref}
\hypersetup{colorlinks,linkcolor={blue},citecolor={blue},urlcolor={blue}}

\newtheorem{theorem}{Theorem}[section]
\newtheorem{proposition}[theorem]{Proposition}
\newtheorem{lemma}[theorem]{Lemma}
\newtheorem{corollary}[theorem]{Corollary}
\theoremstyle{definition}

\theoremstyle{remark}
\newtheorem{remark}[theorem]{Remark}

\numberwithin{equation}{section}

\usepackage{color}

\title[Dimension free weak-type endpoint estimates ...]
{Dimension free weak-type endpoint estimates for the vectors of the Dunkl--Riesz transform}

\author[S. Mukherjee]{Suman Mukherjee}

\address{Department of Mathematics, Indian Institute of Technology Bombay, Powai, Mumbai--400076, India.}
\email{sumanmukherjee822@gmail.com}


\begin{document}

\begin{abstract} Let $R\subset\mathbb{R}^d$ be a reduced root system, $G$ the associated finite reflection group, and $k\ge0$ a $G$-invariant multiplicity function. We develop a Dunkl analogue of the obstacle/partial-balayage method of Ouyang, Spector, and Stockdale (\href{https://arxiv.org/abs/2608.18068}{https://arxiv.org/abs\\
/2608.18068}) for the Euclidean fractional Laplacian. For $0<s<2$, nonnegative $f\in L^1(\nu_k)\cap L^2(\nu_k)$, and $\lambda>0$, we obtain a decomposition \[ f=\mu+(-\Delta_k)^{s/2}u, \qquad 0\le\mu\le\lambda, \] with $\mu=\lambda$ on $\Omega=\{u>0\}$ and \[ \lambda\nu_k(\Omega)\le\|f\|_{L^1(\nu_k)}. \] As an application, we prove that the vector Dunkl--Riesz transform $\mathcal R_k=\nabla_k(-\Delta_k)^{-1/2}$ is of weak type $(1,1)$ with constant at most $(M_k+2)$, where \[ M_k=\#\{\alpha\in R_+:k(\alpha)>0\}. \] For $G$-invariant functions, the reflection terms vanish and the same argument gives the universal constant $2$. We further establish a dimension-free weak-type $(1,1)$ estimate for the Dunkl--Schr\"odinger Riesz transform. 
\end{abstract}
\maketitle

\section{Introduction and main results}

\subsection{Classical motivation}

The Riesz transforms occupy a central place in harmonic analysis, connecting singular
integral theory, Sobolev spaces, potential theory, and probabilistic methods.  In
Euclidean space $\mathbb{R}^d$, the $j$th Riesz transform is the Fourier multiplier
\[
   \widehat{R_jf}(\xi)=-i\frac{\xi_j}{|\xi|}\widehat f(\xi),
\]
and the full vector transform $\mathcal{R}f=(R_1f,\dots,R_df)$ satisfies strong $L^p$ estimates,
$1<p<\infty$, with constants which can be chosen independently of the dimension; see,
for example, \cite{Stein1983,BanuelosWang1995}.  The endpoint $p=1$ is substantially
more delicate.  A question raised by Stein at the 1986 International Congress of
Mathematicians asked whether the weak-type $(1,1)$ norm of the Riesz transforms admits
a bound independent of the dimension \cite{Stein1987}.  Before the recent work of Ouyang, Spector, and
Stockdale, the best general dimensional estimate for an individual Riesz transform was
of order $\log (d)$, due to Janakiraman \cite{Janakiraman2004}; see also
\cite{SpectorStockdale2021} for related progress and historical discussion.

Ouyang, Spector, and Stockdale \cite{OSS2026} gave an affirmative answer in a stronger
vector-valued form: for every $f\in L^1(\mathbb{R}^d)$,
\begin{equation}\label{eq:classical-OSS}
   \|\mathcal{R}f\|_{L^{1,\infty}(\mathbb{R}^d;\ell^2)}\le 2\|f\|_{L^1(\mathbb{R}^d)}.
\end{equation}
The decisive point in their argument is not a
refinement of the classical Calder\'on--Zygmund decomposition.  Instead, they construct
a decomposition adapted to the structural identity
\[
   \mathcal{R}=\nabla(-\Delta)^{-1/2}.
\]
More precisely, a nonnegative function $f$ is decomposed at height $\lambda$ as
\[
   f=\mu+(-\Delta)^{1/2}u,
\]
where $0\le \mu\le\lambda$, the mass of $\mu$ agrees with that of $f$, and the
``potential part'' satisfies $\nabla u=0$ outside the positivity set
$\Omega=\{u>0\}$.  The set $\Omega$ has size at most $\lambda^{-1}\|f\|_1$.  Thus the
bad part is localized geometrically rather than estimated by annular kernel
integrations.  This is precisely the feature that makes the constant in
\eqref{eq:classical-OSS} dimension-free.  The decomposition belongs naturally to the
theory of obstacle problems and variational inequalities and is closely related to
partial balayage and the divisible-sandpile odometer; see
\cite{KinderlehrerStampacchia1980,ServadeiValdinoci2013,OSS2026}.

The purpose of the present paper is to investigate how far this new endpoint mechanism
survives in rational Dunkl analysis.  The answer is twofold.  The obstacle and
Lewy--Stampacchia part of the argument extends remarkably well to the fractional Dunkl
Laplacian, but the final localization step detects a genuinely Dunkl phenomenon: the
first-order Dunkl operators are differential--difference operators and therefore see
reflected values of the potential.  This produces a reflection enlargement of the
positivity set for general data, while the Euclidean localization is recovered exactly
for reflection-invariant data.

\subsection{Riesz transforms in the rational Dunkl setting}

Dunkl analysis originated in the work of C.~F.~Dunkl \cite{Dunkl1989}, who introduced
a commuting family of first-order differential--difference operators attached to a
finite reflection group.  The associated transform theory, developed in particular by
de Jeu \cite{deJeu1993} and R\"osler \cite{Rosler2003}, provides a deformation of
Euclidean Fourier analysis which retains a Fourier multiplier calculus while encoding
the geometry of a root system.

Let $R\subset \mathbb{R}^{d}\setminus\{0\}$ be a reduced root system. Thus, $R$ is a finite set of nonzero vectors, called roots, which is invariant under reflection across the hyperplane orthogonal to each of its elements. More precisely, for $\alpha\in R$, define
\[
\sigma_{\alpha}(x)
   =x-2\frac{\langle x,\alpha\rangle}{|\alpha|^{2}}\alpha,
   \]
where $\langle\cdot,\cdot\rangle$ denotes the standard Euclidean inner product on $\mathbb{R}^{d}$. The root system condition requires that $\sigma_{\alpha}(R)=R$,
for all $\alpha\in R$. The assumption that $R$ is reduced means that no nontrivial scalar multiple of a root is again a root, that is, for every $\alpha\in R$,
\[
R\cap \mathbb{R}\alpha=\{\alpha,-\alpha\}.
\]
We fix a positive subsystem $R_{+}\subset R$, obtained by choosing exactly one element from each pair $\{\alpha,-\alpha\}$. Thus,
\[
R=R_{+}\cup(-R_{+}),
\qquad
R_{+}\cap(-R_{+})=\varnothing.
\]
The finite reflection group associated with $R$ is $G=\langle \sigma_{\alpha}:\alpha\in R\rangle$, namely, the subgroup of the orthogonal group $O(d)$ generated by the reflections $\sigma_{\alpha}$, $\alpha\in R$. For each $\alpha\in R$, the reflection $\sigma_{\alpha}$ fixes pointwise the hyperplane $H_{\alpha}
   =\{x\in\mathbb{R}^{d}:\langle x,\alpha\rangle=0\}$ and maps $\alpha$ to $-\alpha$. Throughout, we use the standard normalization $|\alpha|^{2}=2$ for all $\alpha\in R$. Let
$k:R\to[0,\infty)$ be a $G$-invariant multiplicity function and set
\[
  \gamma_k:=\sum_{\alpha\in R_+}k(\alpha),
  \qquad d_k:=d+2\gamma_k.
\]
The natural $G$-invariant weight and measure are
\[
  w_k(x):=\prod_{\alpha\in R_+}|\langle\alpha,x\rangle|^{2k(\alpha)},
  \qquad d\nu_k(x):=c_k^{-1}w_k(x)\,dx,
\]
where the normalization $c_k$ is given by
\[c_k=\int_{\mathbb{R}^d} e^{-\frac{|x|^2}{2}}\,w_k(x)\,dx.\]
The number $d_k$ plays the role of a homogeneous dimension, in
the sense that
\[
\nu_k(rE)=r^{d_k}\nu_k(E),\qquad \text{ for any } E\subseteq \mathbb{R}^d\text{ and for any }r>0.
\]
For $\xi\in\mathbb{R}^d$, the Dunkl operator is given by
\begin{equation}\label{eq:intro-Dunkl-op}
  T_\xi f(x)=\partial_\xi f(x)
  +\sum_{\alpha\in R_+}k(\alpha)\langle\alpha,\xi\rangle
       \frac{f(x)-f(\sigma_\alpha x)}{\langle\alpha,x\rangle}.
\end{equation}
Writing $T_j=T_{e_j}$, we set
\[
  \nabla_k f=(T_1f,\dots,T_df),
  \qquad \Delta_k=\sum_{j=1}^dT_j^2,
  \qquad \text{and } A_k:=-\Delta_k.
\]
The Dunkl transform $\mathcal F_k$ (see Section \ref{subsec: Dunkl transform} for the definition) diagonalizes these operators; that is,
\[
  \mathcal F_k(T_jf)(\xi)=i\xi_j\mathcal F_kf(\xi)
  \ \text{and }
  \mathcal F_k(A_k^{s/2}f)(\xi)=|\xi|^s\mathcal F_kf(\xi), \ \text{for }s \in \mathbb{R}.
\]
Accordingly, the vector of the Dunkl--Riesz transform is defined by
\[
  \mathcal R_k f=\nabla_k A_k^{-1/2}f
  =(R_{k,1}f,\dots,R_{k,d}f),
\]
and Dunkl-Plancherel theorem gives the exact $L^2$ identity
\begin{equation}\label{eq:L2-isometry}
   \|\mathcal R_k f\|_{L^2(\nu_k;\ell^2)}=\|f\|_{L^2(\nu_k)}.
\end{equation}

The $L^p$ theory of these transforms has developed over the last two decades.
Thangavelu and Xu \cite{ThangaveluXu2007} introduced and studied Dunkl analogues of
Riesz transforms and Riesz potentials, obtaining boundedness in the special situation where $G\equiv \mathbb{Z}^2_d$.
Amri and Sifi \cite{AmriSifi2012} subsequently proved $L^p(\nu_k)$ boundedness of each
Riesz transform for the full range $1<p<\infty$ and arbitrary finite reflection groups.
Their proof adapts Calder\'on--Zygmund theory to the orbit geometry of the Dunkl
setting; in particular, it also yields componentwise weak type $(1,1)$, with constants
coming from the underlying Dunkl singular-integral structure.  Endpoint Hardy-space
theory was developed further by Anker, Dziuba\'nski, and Hejna
\cite{AnkerDziubanskiHejna2019}, who characterized the rational Dunkl Hardy space
$H^1$ by Riesz transforms, maximal functions, square functions, and atoms.

The question of dimension-free strong estimates for the vectors of Dunkl--Riesz transform was
addressed by Hejna \cite{Hejna2023}.  Using a Bellman-function argument adapted to
Dunkl operators, she obtained strong $L^p$ estimates whose dependence on the ambient
dimension is removed; for general data the stated bound retains dependence on the
total multiplicity, whereas for $G$-invariant data the bound is independent of the
root system and the multiplicity function. More precisely, she proved that, for $1<p<\infty$,
\begin{equation}\label{hejna results1}
 \left\lVert \mathcal{R}_{k} f\right\rVert_{L^p(\nu_{k};\ell^2)}
 \leq 144\big(\max\big\{p, \frac{p}{p-1}\big\} -1\big)\big(2\sum_{\alpha\in R_+}k(\alpha)+ 2^7\big)
 \left\lVert f\right\rVert_{L^p(\nu_{k})},
 \end{equation}
and, if $f$ is $G$-invariant, then
\begin{equation}\label{hejna results2}
 \left\lVert \mathcal{R}_{k} f\right\rVert_{L^p(\nu_{k};\ell^2)}
 \leq 144\big(\max\big\{p, \frac{p}{p-1}\big\} -1\big)\left\lVert f\right\rVert_{L^p(\nu_{k})}.
 \end{equation}

Thus, prior to the present work, the Dunkl framework already provided both
dimension-dependent Calder\'on--Zygmund endpoint weak-type estimates and
dimension-free strong-type estimates for $1<p<\infty$. Our objective is to
determine whether a dimension-free weak-type estimate can also be established
for the Dunkl Riesz transform at $p=1$. More precisely, we ask whether the new
obstacle-decomposition mechanism underlying \eqref{eq:classical-OSS} can be
transferred to the setting of Dunkl analysis and what conclusions this approach
yields at the weak endpoint.

\subsection{Fractional Dunkl Laplacians and the obstacle mechanism}

We prove that, for $0<s<2$ the quadratic form associated with $A_k^{s/2}$ has a positive
jump-kernel representation of the schematic form
\begin{equation}\label{eq:intro-jump-form}
  \mathcal E_{k,s}(v,w)
   =\frac12\iint_{\mathbb{R}^d\times\mathbb{R}^d}
       (v(x)-v(y))(w(x)-w(y))
       \mathcal J_{k,s}(x,y)\,d\nu_k(x)d\nu_k(y),
\end{equation}
where $\mathcal J_{k,s}(x,y)\ge0$ is symmetric.  This positivity is the key
structural substitute for the classical Gagliardo form.  It implies the Markov
property under $1$-Lipschitz truncations and, consequently, the sign inequalities
needed in the Lewy--Stampacchia argument.

This observation suggests that the variational core of \cite{OSS2026} is not tied to
translation invariance.  We therefore minimize, for nonnegative
$f\in L^1(\nu_k)\cap L^2(\nu_k)$ and $\lambda>0$, the functional
\begin{equation}\label{eq:intro-energy}
  J_{k,s,\lambda}(v)
   :=\frac12\mathcal E_{k,s}(v)
      -\int_{\mathbb{R}^d}(f-\lambda)v\,d\nu_k
\end{equation}
over the positive cone in the natural fractional Dunkl energy space.  The minimizer
$u$ satisfies a variational inequality, and the distribution
\[
   \eta:=A_k^{s/2}u-(f-\lambda)
\]
is a positive measure.  Testing the variational inequality with the positive part of
$u-\varepsilon\varphi$ yields the Dunkl Lewy--Stampacchia estimate
\begin{equation}\label{eq:intro-LS}
   0\le \eta\le (\lambda-f)^+\le\lambda.
\end{equation}
Defining $\mu=\lambda-\eta$ then produces a mass-preserving partial-balayage
decomposition
\[
   f=\mu+A_k^{s/2}u,
   \qquad 0\le\mu\le\lambda.
\]
As in the Euclidean argument, complementarity gives $\mu=\lambda$ on
$\Omega:=\{u>0\}$ and hence
\[
   \lambda\nu_k(\Omega)\le \|f\|_{L^1(\nu_k)}.
\]
\subsection{Main results}
Let $H_k^1$ be the Sobolev space defined by
\[
H_k^1
=
\left\{
v\in L^2(\nu_k):
\int_{\mathbb{R}^d}
(1+|\xi|^2)
|\mathcal{F}_k v(\xi)|^2\,d\nu_k(\xi)
<\infty
\right\}.
\]
The statement of our first main theorem is as follows.

\begin{theorem}\label{thm:decomp}
Let $0<s<2$, $\lambda>0$, and let
$f\in L^1(\nu_k)\cap L^2(\nu_k)$ be nonnegative.  Then there exist nonnegative
functions $\mu\in L^1(\nu_k)\cap L^\infty(\nu_k)$ and $u\in H_k^s\cap L^1(\nu_k)$, such that
\begin{equation}\label{eq:decomp}
  f=\mu+A_k^{s/2}u
\end{equation}
almost everywhere and
\begin{enumerate}[label=\textup{(\roman*)}]
  \item $0\le\mu\le\lambda$ almost everywhere and
  \[
     \|\mu\|_{L^1(\nu_k)}=\|f\|_{L^1(\nu_k)};
  \]
  \item if $\Omega:=\{x:u(x)>0\}$, then
  \[
     \mu=\lambda\quad\text{almost everywhere on }\Omega;
  \]
  \item
  \[
     \lambda\nu_k(\Omega)\le\|f\|_{L^1(\nu_k)}.
  \]
\end{enumerate}
If $s\ge1$, then $u\in H_k^1$.  Moreover, with
\begin{equation}\label{eq:omega-sharp}
  \Omega^\sharp
  :=\Omega\cup\bigcup_{\substack{\alpha\in R_+\\k(\alpha)>0}}
       \sigma_\alpha(\Omega),
\end{equation}
we have
\begin{equation}\label{eq:gradk-support}
  \nabla_k u=0\quad\text{almost everywhere on }\mathbb{R}^d\setminus\Omega^\sharp.
\end{equation}
Consequently,
\begin{equation}\label{eq:omega-sharp-measure}
  \nu_k(\Omega^\sharp)
    \le (M_k+1)\nu_k(\Omega),
  \qquad
  M_k:=\#\{\alpha\in R_+:k(\alpha)>0\}.
\end{equation}
If, in addition, $f$ is $G$-invariant, then $u$ and $\mu$ may be chosen
$G$-invariant and in that case $\Omega^\sharp=\Omega$ up to null sets.
\end{theorem}

Theorem~\ref{thm:decomp} shows that essentially the entire obstacle-theoretic part of
the Euclidean proof survives in the Dunkl setting.  The one place where a literal
conversion fails is also the place where the differential--difference nature of the
Dunkl operators becomes unavoidable.  Indeed, in the Euclidean proof one uses the
Stampacchia property $\nabla u=0$ almost everywhere on $\{u=0\}$ to conclude that the
potential term is supported in $\Omega$.  For a Dunkl derivative, however,
\begin{equation}\label{eq:intro-reflection-obstruction}
  T_j u(x)=\partial_j u(x)
   +\sum_{\alpha\in R_+}k(\alpha)\alpha_j
      \frac{u(x)-u(\sigma_\alpha x)}{\langle\alpha,x\rangle}.
\end{equation}
Thus $u(x)=0$ and $\nabla u(x)=0$ do not force $T_j u(x)=0$: one must also exclude
points whose reflections lie in $\Omega$.  This leads exactly to the reflected set
$\Omega^\sharp$ in \eqref{eq:omega-sharp}.  Since $\nu_k$ is reflection invariant,
\[
   \nu_k(\Omega^\sharp)\le(M_k+1)\nu_k(\Omega).
\]
Combining this with \eqref{eq:L2-isometry} and the bounded part of the obstacle
decomposition gives our endpoint consequence:

\begin{theorem}\label{thm:weak}
For every real-valued $f\in L^1(\nu_k)$,
\begin{equation}\label{eq:weak-main}
  \|\mathcal R_k f\|_{L^{1,\infty}(\nu_k;\ell^2)}
  :=\sup_{\lambda>0}
      \lambda\,\nu_k\bigl(\{x:|\mathcal R_k f(x)|_{\ell^2}>\lambda\}\bigr)
  \le (M_k+2)\|f\|_{L^1(\nu_k)}.
\end{equation}
If $f$ is $G$-invariant, then
\begin{equation}\label{eq:weak-inv}
  \|\mathcal R_k f\|_{L^{1,\infty}(\nu_k;\ell^2)}
  \le 2\|f\|_{L^1(\nu_k)}.
\end{equation}
\end{theorem}

When $k\equiv 0$, we have $M_k=0$, $\nabla_k=\nabla$, $\nu_k$ coincides with the Lebesgue measure, and $\mathcal{R}_k=\mathcal{R}$. Consequently, Theorem~\ref{thm:weak} recovers \eqref{eq:classical-OSS}, with the constant $2$.  More significantly, the same constant persists for every
root system and every nonnegative multiplicity on the $G$-invariant subspace.  For
general data our proof produces the explicit reflection loss $M_k$.  We do not claim
that this loss is optimal.  Rather, it identifies sharply where the obstacle method
ceases to behave as in Euclidean space.  Removing the reflected enlargement in the
arbitrary-data problem would require an additional cancellation or localization
principle not contained in the direct variational argument.

The preceding endpoint estimate, combined with the Marcinkiewicz interpolation theorem, also yields strong $L^p$ bounds with constants independent of the dimension. Although the strong $L^p$ boundedness of Dunkl--Riesz transforms is already known by singular-integral methods, the interpolation argument below makes explicit the additional information furnished by the endpoint theorem. In particular, in the range $1<p\leq 2$, it recovers a result similar to the dimension-free estimate obtained by Hejna, as stated in \eqref{hejna results1} and \eqref{hejna results2}.

\begin{corollary}
\label{thm:interpolated-vector-bound}
Let $1<p\leq2$. There is a constant $C_p$, depending only on $p$, such that
\begin{equation}
 \left\lVert \mathcal{R}_{k} f\right\rVert_{L^p(\nu_{k};\ell^2)}
 \leq C_p(M_k+2)^{\frac2p-1}
 \left\lVert f\right\rVert_{L^p(\nu_{k})}
 \label{eq:interpolated-general}
\end{equation}
for any real-valued $f\in L^p(\nu_{k})$. If $f$ is $G$-invariant, then
\begin{equation}
 \left\lVert \mathcal{R}_{k} f\right\rVert_{L^p(\nu_{k};\ell^2)}
 \leq C_p2^{\frac2p-1}\left\lVert f\right\rVert_{L^p(\nu_{k})}.
 \label{eq:interpolated-invariant}
\end{equation}
In particular, the constant in \eqref{eq:interpolated-invariant} is
independent of $d$, $R$, and $k$.
\end{corollary}
\subsection{Novelty and organization of the paper}

The main contributions of the paper can be summarized as follows.
\begin{enumerate}[label=\textup{(\roman*)}]
  \item Using the properties of the Dunkl heat kernel and the subordination formula,
we derive a positive jump-kernel representation for the quadratic form
associated with $A_k^{s/2}$.

  \item We formulate and prove a whole-space fractional Dunkl obstacle/partial-balayage
  decomposition with $L^1$ mass preservation, an $L^\infty$ cap, complementarity on
  the positivity set, and quantitative control of that set.

  \item We derive a Lewy--Stampacchia estimate for the fractional Dunkl energy using
  the Markov property of the positive jump form.  This is the step that permits the
  obstacle argument to pass from the Euclidean fractional Laplacian to the Dunkl
  fractional Laplacian without relying on an explicit translation formula.

  \item We isolate a reflection-localization lemma for $H_k^1$ functions.  It shows
  that the natural Dunkl gradient of the obstacle potential vanishes outside the
  reflected enlargement $\Omega^\sharp$, and it explains why the Euclidean support
  argument cannot simply be copied for arbitrary data.

  \item We obtain the vector weak-type estimate \eqref{eq:weak-main} by combining the
  obstacle decomposition with the exact $L^2$ isometry of the vector Dunkl--Riesz
  transform.  On $G$-invariant data the reflection obstruction disappears and the
  constant $2$ is recovered exactly.

  \item As an application, we establish a dimension-free weak-type $(1,1)$ estimate
for the Dunkl--Schr\"odinger Riesz transform, which, to the best of our
knowledge, is also new in the classical setting.
\end{enumerate}

The paper is organized as follows. Section~\ref{sec:Prel} recalls the basic
facts concerning the Dunkl transform, weighted Sobolev spaces, and the
Stampacchia property that will be used throughout the paper.
Section~\ref{sec: Energy, min and var} develops the fractional Dunkl energy
and establishes the positivity, truncation, and interpolation properties of
the associated quadratic form. In the same section, we minimize
\eqref{eq:intro-energy} over the positive cone and derive the corresponding
variational inequality and energy identity.
Section~\ref{sec:Lewy--Stampacchia} is devoted to the Dunkl
Lewy--Stampacchia estimate, mass conservation, regularity, and the
complementarity relation. Section~\ref{sec: man thm proof} establishes the
reflection-localization principle and proves Theorem~\ref{thm:weak}.
Finally, in Section~\ref{sec: application}, we apply these results to obtain
a dimension-free weak-type $(1,1)$ estimate for the
Dunkl--Schr\"odinger Riesz transform.

\section{Dunkl transform and weighted Sobolev preliminaries}\label{sec:Prel}

We collect the structural facts needed later.  Standard references for the Dunkl
transform and Dunkl operators include Dunkl \cite{Dunkl1989}, de Jeu \cite{deJeu1993}, and R\"osler \cite{Rosler2003}.

\subsection{The Dunkl transform and fractional powers}\label{subsec: Dunkl transform}

Let $E_k(\cdot,\cdot)$ denote the Dunkl kernel. With the normalization fixed earlier, the Dunkl transform is defined by
\[
  \mathcal{F}_{k}f(\xi)=\int_{\mathbb{R}^d} f(x)E_k(-i\xi,x)\,d\nu_k(x),
\]
for $f\in L^1(\nu_k)$.  Since $\left\lvert E_k(-i\xi,x)\right\rvert\le1$ for real
$x,\xi$, we have
\begin{equation}\label{eq:F-L1}
  \left\lvert \mathcal{F}_{k}f(\xi)\right\rvert\le\left\lVert f\right\rVert_{L^1(\nu_k)}.
\end{equation}
The transform extends to a unitary operator on $L^2(\nu_k)$, satisfying
\[
\|\mathcal{F}_k f\|_{L^2(\nu_k)}=\|f\|_{L^2(\nu_k)},
\]
and maps the Schwartz space $\mathcal{S}(\mathbb{R}^d)$ onto itself.

For $s \in \mathbb{R}$, the inhomogeneous Dunkl Sobolev space $H_k^s$ is defined by
\[
H_k^s
=
\left\{
v\in L^2(\nu_k):
\int_{\mathbb{R}^d}
(1+|\xi|^2)^s
|\mathcal{F}_k v(\xi)|^2\,d\nu_k(\xi)
<\infty
\right\}.
\]
For $v\in L^1(\nu_k)$ define the (possibly infinite) quadratic energy by
\begin{equation}\label{eq:Ealpha-def}
  \mathcal{E}_{k,s}(v)
   :=\int_{\mathbb{R}^d}\left\lvert \xi\right\rvert^{s}\left\lvert \mathcal{F}_{k}v(\xi)\right\rvert^2\,d\nu_k(\xi).
\end{equation}
If $v,w$ have finite energy, the associated bilinear form is obtained by polarization, that is,
\[
  \mathcal{E}_{k,s}(v,w)
   :=\int_{\mathbb{R}^d}\left\lvert \xi\right\rvert^{s}
          \mathcal{F}_{k}v(\xi)\overline{\mathcal{F}_{k}w(\xi)}\,d\nu_k(\xi).
\]
We use the energy space
\begin{equation}\label{eq:Xalpha}
  X_{k,s}:=
  \{v\in L^1(\nu_k):\mathcal{E}_{k,s}(v)<\infty\}.
\end{equation}
The following proposition provides a positive jump-kernel representation of
the quadratic form generated by $A_k^{s/2}$ for every $0<s<2$.

\begin{proposition}
\label{prop:gagliardo-dunkl}
Let $0<s<2$. There exists a nonnegative symmetric
kernel $J_{k,s}\colon \mathbb{R}^d\times\mathbb{R}^d
  \rightarrow [0,\infty]$ such that, for all $v,w\in\mathcal{S}(\mathbb{R}^d)$,
\begin{equation}\label{eq:gagliardo-dunkl}
  \mathcal{E}_{k,s}(v,w)
  =
  \frac{C_{k,d,s}}{2}
  \iint_{\mathbb{R}^d\times\mathbb{R}^d}
  (v(x)-v(y))
  \overline{(w(x)-w(y))}
  J_{k,s}(x,y)\,d\nu_k(x)d\nu_k(y).
\end{equation}
The identity extends by closure to all $v,w$ in the form domain of
$\mathcal{E}_{k,s}$. In particular,
\[
  \mathcal{E}_{k,s}(v)
  =
  \frac{C_{k,d,s}}{2}
  \iint_{\mathbb{R}^d\times\mathbb{R}^d}
  |v(x)-v(y)|^2
  J_{k,s}(x,y)\,d\nu_k(x)d\nu_k(y).
\]
For real-valued functions, the complex conjugation in
\eqref{eq:gagliardo-dunkl} may be omitted.
\end{proposition}

\begin{proof}
We follow the idea in \cite[Section 6]{AnoopParui2019}. Let $(H_t^k)_{t>0}$ denote the Dunkl heat semigroup, characterized by
\[
  \mathcal{F}_k(H_t^k f)(\xi)
  =e^{-t|\xi|^2}\mathcal{F}_k f(\xi).
\]
The by \cite[Definition~4.4 and Lemma~4.5]{Rosler1998}, it is a symmetric Markov semigroup on $L^2(\nu_k)$ and admits a
measurable heat kernel $h_t^k(x,y)$ satisfying
\[
  H_t^k f(x)
  =\int_{\mathbb{R}^d}h_t^k(x,y)f(y)\,d\nu_k(y),
\]
together with
\begin{align}\label{prop heat ker}
  h_t^k(x,y)=h_t^k(y,x)\geq0,
  \int_{\mathbb{R}^d}h_t^k(x,y)\,d\nu_k(y)=1.
\end{align}
Define
\begin{equation}\label{eq:Jka-def}
  J_{k,s}(x,y)
  :=
  \int_0^\infty h_t^k(x,y)\,
  \frac{dt}{t^{1+s/2}}.
\end{equation}
The kernel $J_{k,s}$ is therefore nonnegative and symmetric.

For $0<s<2$, the scalar identity
\begin{equation}\label{eq:balakrishnan-scalar}
  \lambda^{s/2}
  =
  c_s
  \int_0^\infty
  \bigl(1-e^{-t\lambda}\bigr)
  \frac{dt}{t^{1+s/2}},
  \qquad \lambda\geq0,
\end{equation}
holds with
\[
  c_s
  :=
  \frac{s}{2\Gamma(1-s/2)}
  =
  -\frac{1}{\Gamma(-s/2)}.
\]
Consequently, Dunkl--Plancherel gives, for
$v,w\in\mathcal{S}(\mathbb{R}^d)$,
\begin{align}
  \mathcal{E}_{k,s}(v,w)
  &=
  c_s
  \int_0^\infty
  \left\langle (I-H_t^k)v,w\right\rangle_{L^2(\nu_k)}
  \frac{dt}{t^{1+s/2}}.
  \label{eq:energy-semigroup}
\end{align}
The symmetry and conservativeness of the heat kernel imply
\begin{align}
  \left\langle (I-H_t^k)v,w\right\rangle_{L^2(\nu_k)}
  &=
  \frac12
  \iint_{\mathbb{R}^d\times\mathbb{R}^d}
  (v(x)-v(y))
  \overline{(w(x)-w(y))}
  \nonumber\\
  &\hspace{3.5cm}\times
  h_t^k(x,y)\,d\nu_k(x)d\nu_k(y).
  \label{eq:heat-difference-identity}
\end{align}
Indeed, this follows by expanding the product on the right-hand side
and using \eqref{prop heat ker}.

Substituting \eqref{eq:heat-difference-identity} into
\eqref{eq:energy-semigroup} and applying Tonelli's theorem first to
the associated quadratic form, followed by polarization, we obtain
\begin{align*}
  \mathcal{E}_{k,s}(v,w)
  &=
  \frac{c_s}{2}
  \iint_{\mathbb{R}^d\times\mathbb{R}^d}
  (v(x)-v(y))
  \overline{(w(x)-w(y))}
  J_{k,s}(x,y)\,d\nu_k(x)d\nu_k(y).
\end{align*}
Thus \eqref{eq:gagliardo-dunkl} holds. Any normalization constants arising
from $d\nu_k$ or the Dunkl heat kernel may be absorbed into
$C_{k,d,s}$.

Finally, taking $w=v$ shows that
\[
  \mathcal{E}_{k,s}(v)
  =
  \frac{C_{k,d,s}}{2}
  \iint_{\mathbb{R}^d\times\mathbb{R}^d}
  |v(x)-v(y)|^2
  J_{k,s}(x,y)\,d\nu_k(x)d\nu_k(y).
\]
The right-hand side is a nonnegative quadratic form and is therefore
closable. Passing to limits in the corresponding form norm extends
the identity to the closed form domain. The bilinear identity then
follows by polarization.
\end{proof}
 The explicit form of $J_{k,s}(x,y)$ will not be needed; only its
positivity, symmetry, and representation of the spectral form will be used.

We next establish two lemmas concerning properties of the bilinear form
$\mathcal{E}_{k,s}$. The first lemma states that two nonnegative
functions with disjoint supports have a nonpositive interaction with respect
to the fractional Dunkl energy. The second lemma expresses the Markov
property of the energy: applying any $1$-Lipschitz contraction that fixes
the origin cannot increase the energy. In particular, taking the positive
and negative parts preserves the energy space and yields the stated sign
inequalities.

\begin{lemma}\label{lem:disjoint}
Let $v,w\in X_{k,s}$ be nonnegative and satisfy $vw=0$ almost everywhere.
Then
\[
  \mathcal{E}_{k,s}(v,w)\le0.
\]
\end{lemma}

\begin{proof}
For almost every $x,y$,
\[
 (v(x)-v(y))(w(x)-w(y))
 =-v(x)w(y)-v(y)w(x)\le0.
\]
The conclusion then follows immediately from the representation
\eqref{eq:gagliardo-dunkl}.
\end{proof}

\begin{lemma}\label{lem:markov}
If $\Phi:\mathbb{R}\to\mathbb{R}$ is $1$-Lipschitz, $\Phi(0)=0$, and $v$ is real-valued with
$v\in X_{k,s}$, then $\Phi\circ v\in X_{k,s}$ and
\[
  \mathcal{E}_{k,s}(\Phi\circ v)\le\mathcal{E}_{k,s}(v).
\]
In particular $v^+,v^-\in X_{k,s}$ and
\[
  \mathcal{E}_{k,s}(v,v^- )\le-\mathcal{E}_{k,s}(v^-)\le0,
  \qquad
  \mathcal{E}_{k,s}(v,v^+ )\ge \mathcal{E}_{k,s}(v^+)\ge0.
\]
\end{lemma}

\begin{proof}
The first assertion follows directly from
$\left\lvert \Phi(v(x))-\Phi(v(y))\right\rvert\le\left\lvert v(x)-v(y)\right\rvert$ in
\eqref{eq:gagliardo-dunkl}.  The final inequalities follow from
$v=v^+-v^-$ and Lemma~\ref{lem:disjoint}.
\end{proof}

\subsection{An \texorpdfstring{$L^1$}{L1}--energy interpolation estimate}

Let $b_k:=\nu_k(B(0,1))$. By homogeneity, $\nu_k(B(0,r))=b_kr^{d_k}$.
We next establish a lemma which is Dunkl analogue of the
Gagliardo--Nirenberg interpolation inequality.
\begin{lemma}\label{lem:GN}
For $0<s<2$ and $v\in X_{k,s}$,
\begin{equation}\label{eq:GN}
  \left\lVert v\right\rVert_{L^2(\nu_k)}^2
  \le C_{k,s}
   \left\lVert v\right\rVert_{L^1(\nu_k)}^{\frac{2s}{{d_k}+s}}
   \mathcal{E}_{k,s}(v)^{\frac{{d_k}}{{d_k}+s}},
\end{equation}
where one may take
\[
 C_{k,s}
 =\frac{{d_k}+s}{s}
   b_k^{\frac{s}{{d_k}+s}}
   \left(\frac{s}{{d_k}}\right)^{\frac{{d_k}}{{d_k}+s}}.
\]
\end{lemma}

\begin{proof}
For $r>0$, Plancherel, \eqref{eq:F-L1}, and \eqref{eq:Ealpha-def} give
\begin{align*}
 \left\lVert v\right\rVert_2^2
 &=\int_{\left\lvert \xi\right\rvert<r}\left\lvert \mathcal{F}_{k}v(\xi)\right\rvert^2\,d\nu_k(\xi)
   +\int_{\left\lvert \xi\right\rvert\ge r}\left\lvert \mathcal{F}_{k}v(\xi)\right\rvert^2\,d\nu_k(\xi)\\
 &\le b_kr^{d_k}\left\lVert v\right\rVert_1^2+r^{-s}\mathcal{E}_{k,s}(v).
\end{align*}
Minimizing the right-hand side in $r$ yields \eqref{eq:GN}.
\end{proof}
\begin{remark}
This interpolation estimate  implies $X_{k,s}\subset L^2(\nu_k)$. 
\end{remark}
\subsection{First-order Dunkl energy and level sets}

It is known that the Dunkl
operators are skew-symmetric with respect to $d\nu_k$, namely,
\[
  \int_{\mathbb{R}^d}T_jf(x)\,\overline{g(x)}\,d\nu_k(x)
  =
  -\int_{\mathbb{R}^d}f(x)\,\overline{T_jg(x)}\,d\nu_k(x).
\]
On the other hand, by the definition of the carré du champ,
\[
  \Gamma_k(v)
  =
  \frac12\bigl(\Delta_k|v|^2
  -2\operatorname{Re}(\overline{v}\,\Delta_kv)\bigr).
\]
Now, for smooth compactly supported $v$, the carré-du-champ identity is (see, for example, \cite[Lemma 3.1]{Velicu2020})
\begin{equation}\label{eq:carre}
  \Gamma_k(v)(x)
  =\left\lvert \nabla v(x)\right\rvert^2
   +\sum_{\alpha\in R_+}k(\alpha)
      \frac{\left\lvert v(x)-v(\sigma_\alpha x)\right\rvert^2}{\left\langle \alpha,x\right\rangle^2}.
\end{equation}
Integration by parts gives
\begin{equation}\label{eq:first-energy}
  \sum_{j=1}^d\left\lVert T_jv\right\rVert_2^2
  =\int_{\mathbb{R}^d}\Gamma_k(v)\,d\nu_k
  =\int_{\mathbb{R}^d}\left\lvert \xi\right\rvert^2\left\lvert \mathcal{F}_{k}v(\xi)\right\rvert^2\,d\nu_k(\xi).
\end{equation}
In particular,
\begin{equation}\label{eq:ordinary-grad-bound}
  \int\left\lvert \nabla v\right\rvert^2\,d\nu_k
  \le\sum_{j=1}^d\left\lVert T_jv\right\rVert_2^2.
\end{equation}
By closure, these formulas remain valid on the first-order Dunkl Sobolev space
$H_k^1$. We now turn to the weighted Stampacchia property and the
reflection-localization principle.

\begin{lemma}\label{lem:stamp}
If $v\in H_k^1$ is real-valued, then
\[
  \nabla v=0\qquad \nu_k\text{-a.e. on }\{v=0\}.
\]
More generally, $\nabla v=0$ $\nu_k$-almost everywhere on every level set
$\{v=c\}$.
\end{lemma}

\begin{proof}
The union of the reflection hyperplanes $H_{\alpha}
   =\{x\in\mathbb{R}^{d}:\langle x,\alpha\rangle=0\}$, has Lebesgue, hence $\nu_k$, measure zero.
On a compact set contained in a connected component of the complement of those
hyperplanes, the weight $w_k$ is bounded above and below by positive constants.
The estimate \eqref{eq:ordinary-grad-bound} therefore implies that $v$ belongs to the
usual local Sobolev space $W^{1,2}_{\mathrm{loc}}$ there.  The classical Stampacchia
level-set theorem gives $\nabla v=0$ almost everywhere on $\{v=c\}$ inside each such
compact set.  Exhausting the complement of the hyperplanes proves the claim.
\end{proof}

\begin{lemma}\label{lem:reflection-localization}
Let $v\in H_k^1$ and $v$ is real-valued. Also, let $\Omega=\{v\ne0\}$, and define
\[
  \Omega^\sharp
  =\Omega\cup\bigcup_{\substack{\alpha\in R_+\\k(\alpha)>0}}
       \sigma_\alpha(\Omega).
\]
Then $\nabla_{k} v=0\   \nu_k\text{-a.e. on  }\mathbb{R}^d\setminus\Omega^\sharp$. Moreover,
\[
  \nu_k(\Omega^\sharp)\le(M_k+1)\nu_k(\Omega).
\]
If $\Omega$ is $G$-invariant, then $\Omega^\sharp=\Omega$.
\end{lemma}

\begin{proof}
Take $x\notin\Omega^\sharp$ away from the reflection hyperplanes and from a null set
on which Lemma~\ref{lem:stamp} may fail.  Then $v(x)=0$,
$\nabla v(x)=0$, and $v(\sigma_\alpha x)=0$ for every $\alpha$ with $k(\alpha)>0$.
The differential-difference formula for $T_j$ therefore gives $T_jv(x)=0$ for every
$j$.  The measure estimate follows from $G$-invariance of $\nu_k$ and the union bound.
The last assertion is immediate.
\end{proof}

\section{Energies, minimizers, and variational inequalities}\label{sec: Energy, min and var}

Fix throughout this section $0<s<2$, $\lambda>0$, and a nonnegative
$f\in L^1(\nu_k)\cap L^2(\nu_k)$.  Define
\begin{equation}\label{eq:J}
  J_{\lambda}^{k,s}(v)
   =\frac12\mathcal{E}_{k,s}(v)-\int_{\mathbb{R}^d}(f-\lambda)v\,d\nu_k,
\end{equation}
for $v\in X_{k,s}$, and let
\[
  \mathcal{K}_{k,s}:=\{v\in X_{k,s}:v\ge0\text{ a.e.}\}.
\]
For $v\in\mathcal{K}_{k,s}$, we have
\begin{equation}\label{eq:J-positive}
  J_{\lambda}^{k,s}(v)
   =\frac12\mathcal{E}_{k,s}(v)-\int fv\,d\nu_k
      +\lambda\left\lVert v\right\rVert_{L^1(\nu_k)}.
\end{equation}
We first prove that the functional $J_{\lambda}^{k,s}$ is
coercive on the admissible set $\mathcal{K}_{k,s}$. More precisely,
$J_{\lambda}^{k,s}(v)$ controls both the fractional Dunkl energy
$\mathcal{E}_{k,s}(v)$ and the $L^1(\nu_k)$-norm of $v$, up to a
fixed constant $C_0$.
\begin{proposition}\label{prop:coercive}
There exists a finite constant
$C_0=C_0(k,d,s,\lambda,\left\lVert f\right\rVert_2)$ such that
\begin{equation}\label{eq:coercive}
  J_{\lambda}^{k,s}(v)
  \ge\frac14\mathcal{E}_{k,s}(v)
      +\frac\lambda2\left\lVert v\right\rVert_{L^1(\nu_k)}-C_0
\end{equation}
for every $v\in\mathcal{K}_{k,s}$.
\end{proposition}

\begin{proof}
By Cauchy--Schwarz and Lemma~\ref{lem:GN},
\[
  \int fv\,d\nu_k
  \le \left\lVert f\right\rVert_2C_{k,s}^{1/2}
      A^{\frac{s}{{d_k}+s}}
      E^{\frac{{d_k}}{2({d_k}+s)}},
\]
where $A=\left\lVert v\right\rVert_1$ and $E=\mathcal{E}_{k,s}(v)$.  The two exponents sum to less than one;
the remaining exponent equals ${d_k}/[2({d_k}+s)]$.  Weighted Young's inequality,
with coefficients chosen to absorb $\lambda A/2$ and $E/4$, gives
\[
  \int fv\,d\nu_k\le\frac\lambda2A+\frac14E+C_0.
\]
Inserting this into \eqref{eq:J-positive} completes the proof.
\end{proof}
We next establish the existence and uniqueness of a minimizer for the
functional $J_{\lambda}^{k,s}$.
\begin{theorem}\label{thm:minimizer}
The functional $J_{\lambda}^{k,s}$ has a unique minimizer
$u\in\mathcal{K}_{k,s}$.
\end{theorem}

\begin{proof}
Since $J_{\lambda}^{k,s}(0)=0$ and Proposition~\ref{prop:coercive} bounds the
functional from below, choose a minimizing sequence $v_m$.  Proposition~\ref{prop:coercive} gives
\[
 \sup_m\bigl(\mathcal{E}_{k,s}(v_m)+\left\lVert v_m\right\rVert_1\bigr)<\infty.
\]
Lemma~\ref{lem:GN} then makes $v_m$ bounded in $L^2(\nu_k)$.  Passing to a
subsequence, $v_m\rightharpoonup u$ weakly in $L^2$.

Testing against nonnegative compactly supported functions shows $u\ge0$.
Likewise, testing against $0\le\phi\le1$ and then exhausting $\mathbb{R}^d$ gives
\[
  \left\lVert u\right\rVert_1\le\liminf_m\left\lVert v_m\right\rVert_1.
\]
The Dunkl transforms converge weakly in $L^2$, and for each $R>0$ the seminorm
\[
  g\mapsto\left(\int_{\left\lvert \xi\right\rvert<R}\left\lvert \xi\right\rvert^{s}\left\lvert g(\xi)\right\rvert^2
        \,d\nu_k(\xi)\right)^{1/2}
\]
is weakly lower semicontinuous.  Letting $R\to\infty$ gives
$\mathcal{E}_{k,s}(u)\le\liminf_m\mathcal{E}_{k,s}(v_m)$.  Since $f\in L^2$, the linear term
$\int fv_m$ converges to $\int fu$, and $u$ is a minimizer.

For uniqueness, the parallelogram identity gives, for $v,w\in\mathcal{K}_{k,s}$,
\[
 J_{\lambda}^{k,s}\left(\frac{v+w}{2}\right)
 =\frac12J_{\lambda}^{k,s}(v)
  +\frac12J_{\lambda}^{k,s}(w)
  -\frac18\mathcal{E}_{k,s}(v-w).
\]
Two minimizers therefore satisfy $\mathcal{E}_{k,s}(v-w)=0$.  Since
$\left\lvert \xi\right\rvert^{s}>0$ away from the singleton $\{0\}$, which has $\nu_k$-measure
zero, Plancherel gives $v=w$.
\end{proof}
We next provide a variational characterization of the minimizer $u$ through a variational
inequality: $u$ is the unique element of the admissible set for which
every admissible variation $v-u$ satisfies the stated energy inequality.
It also gives an energy identity showing that the fractional Dunkl energy
of $u$ is exactly equal to the pairing of $u$ with the forcing term
$f-\lambda$.
\begin{theorem}\label{thm:VI}
The minimizer $u$ satisfies
\begin{equation}\label{eq:VI}
  \mathcal{E}_{k,s}(u,v-u)
   \ge\int_{\mathbb{R}^d}(f-\lambda)(v-u)\,d\nu_k
  \qquad \text{for any }v\in\mathcal{K}_{k,s}.
\end{equation}
Conversely, any $u\in\mathcal{K}_{k,s}$ satisfying \eqref{eq:VI} is the unique minimizer.
Moreover,
\begin{equation}\label{eq:energy-identity}
  \mathcal{E}_{k,s}(u)=\int_{\mathbb{R}^d}(f-\lambda)u\,d\nu_k.
\end{equation}
\end{theorem}

\begin{proof}
For $v\in\mathcal{K}_{k,s}$, the function
$t\mapsto J_{\lambda}^{k,s}(u+t(v-u))$ has a minimum at $t=0$ on $[0,1]$.
Its right derivative at zero yields \eqref{eq:VI}.  Conversely,
\begin{align*}
 J_{\lambda}^{k,s}(v)-J_{\lambda}^{k,s}(u)
 &=\frac12\mathcal{E}_{k,s}(v-u)
   +\mathcal{E}_{k,s}(u,v-u)
   -\int(f-\lambda)(v-u)\,d\nu_k\\
 &\ge0.
\end{align*}
Finally using $v=0$ and $v=2u$ in \eqref{eq:VI} proves \eqref{eq:energy-identity}.
\end{proof}

\section{The Dunkl Lewy--Stampacchia estimate}\label{sec:Lewy--Stampacchia}

Define a distribution $\eta$ on $C_c^\infty(\mathbb{R}^d)$ by
\begin{equation}\label{eq:eta-def}
  \langle\eta,\phi\rangle
  :=\mathcal{E}_{k,s}(u,\phi)-\int_{\mathbb{R}^d}(f-\lambda)\phi\,d\nu_k.
\end{equation}
Equivalently,
\[
  \eta=A_{k}^{s/2}u-(f-\lambda)
\]
in the distributional sense relative to $\nu_k$.
If $\phi\ge0$, then $u+\phi\in\mathcal{K}_{k,s}$ and the variational inequality \eqref{eq:VI} gives
$\langle\eta,\phi\rangle\ge0$.  Thus $\eta$ is a positive Radon measure.
We next establish a Lewy--Stampacchia type estimate.
\begin{theorem}\label{thm:LS}
For every nonnegative $\phi\in C_c^\infty(\mathbb{R}^d)$,
\begin{equation}\label{eq:LS}
  0\le\langle\eta,\phi\rangle
  \le\int_{\mathbb{R}^d}(\lambda-f)^+\phi\,d\nu_k
  \le\lambda\left\lVert \phi\right\rVert_{L^1(\nu_k)}.
\end{equation}
Consequently, $\eta$ is absolutely continuous with respect to $\nu_k$ and its density,
still denoted by $\eta$, satisfies
\begin{equation}\label{eq:eta-density}
  0\le\eta\le(\lambda-f)^+\le\lambda
  \qquad\text{a.e.}
\end{equation}
\end{theorem}

\begin{proof}
Fix $\varepsilon>0$ and put $z=u-\varepsilon\phi$ and $v=z^+$.  By the Markov property (Lemma~ \ref{lem:markov}),
$v\in\mathcal{K}_{k,s}$.  Since $v=z+z^-$,
\[
  v-u=-\varepsilon\phi+z^-.
\]
Substituting into \eqref{eq:VI} gives
\begin{equation}\label{eq:LS-start}
  \varepsilon\langle\eta,\phi\rangle
  \le \mathcal{E}_{k,s}(u,z^-)
      +\int_{\mathbb{R}^d}(\lambda-f)z^-\,d\nu_k.
\end{equation}
Now $u=z+\varepsilon\phi$, so Lemma~\ref{lem:markov} and Cauchy--Schwarz imply
\begin{align*}
 \mathcal{E}_{k,s}(u,z^-)
 &=\mathcal{E}_{k,s}(z,z^-)+\varepsilon\mathcal{E}_{k,s}(\phi,z^-)\\
 &\le-\mathcal{E}_{k,s}(z^-)
      +\varepsilon\mathcal{E}_{k,s}(\phi)^{1/2}\mathcal{E}_{k,s}(z^-)^{1/2}\\
 &\le\frac{\varepsilon^2}{4}\mathcal{E}_{k,s}(\phi).
\end{align*}
Since $u\ge0$,
\[
  0\le z^-=(\varepsilon\phi-u)^+\le\varepsilon\phi,
\]
and hence
\[
  \int(\lambda-f)z^-\,d\nu_k
  \le\varepsilon\int(\lambda-f)^+\phi\,d\nu_k.
\]
Inserting these estimates into \eqref{eq:LS-start}, dividing by $\varepsilon$, and letting
$\varepsilon\rightarrow0^{+}$ proves \eqref{eq:LS}.  The absolute-continuity conclusion follows from domination of the
positive measure $\eta$ by $(\lambda-f)^+\nu_k$.
\end{proof}

Recalling that $\lambda$ is the constant density with respect to $\nu_k$, define
\begin{equation}\label{eq:mu-def}
  \mu=\lambda-\eta.
\end{equation}
Then from \eqref{eq:eta-density}, we have
\begin{equation}\label{eq:mu-basic}
  0\le\mu\le\lambda,
  \qquad
  \mu\ge\min\{\lambda,f\},
\end{equation}
and, also from \eqref{eq:eta-def},
\begin{equation}\label{eq:mu-test}
  \int_{\mathbb{R}^d}\mu\phi\,d\nu_k
   =\int_{\mathbb{R}^d}f\phi\,d\nu_k-\mathcal{E}_{k,s}(u,\phi),
  \qquad \phi\in C_c^\infty(\mathbb{R}^d).
\end{equation}
Our next lemma concerns the conservation of mass, namely, that $\mu$ has the same mass as $f$.
\begin{lemma}\label{lem:mass}
The function $\mu$ belongs to $L^1(\nu_k)$ and
\begin{equation}\label{eq:mass}
  \left\lVert \mu\right\rVert_{L^1(\nu_k)}=\left\lVert f\right\rVert_{L^1(\nu_k)}.
\end{equation}
\end{lemma}

\begin{proof}
Choose a radial, radially nonincreasing $\chi\in C_c^\infty(\mathbb{R}^d)$ such that
$0\le\chi\le1$, $\chi\equiv1$ on $B(0,1)$, and $\operatorname{supp}\chi\subset B(0,2)$.
Set $\chi_\rho(x)=\chi(x/\rho)$.  Then $\chi_\rho\rightarrow 1$ pointwise.
The Dunkl scaling law gives
\[
  \mathcal{F}_{k}(\chi_\rho)(\xi)=\rho^{d_k}\mathcal{F}_{k}\chi(\rho\xi).
\]
Using $\left\lvert \mathcal{F}_{k}u\right\rvert\le\left\lVert u\right\rVert_1$ and changing variables $\zeta=\rho\xi$,
\begin{align*}
 \left\lvert \mathcal{E}_{k,s}(u,\chi_\rho)\right\rvert
 &\le \rho^{-s}\left\lVert u\right\rVert_1
    \int_{\mathbb{R}^d}\left\lvert \zeta\right\rvert^{s}\left\lvert \mathcal{F}_{k}\chi(\zeta)\right\rvert\,d\nu_k(\zeta)
  \rightarrow0.
\end{align*}
The last integral is finite because the Dunkl transform preserves the Schwartz class.
Applying \eqref{eq:mu-test} to $\chi_\rho$ and using monotone convergence yields
\[
 \int\mu\,d\nu_k=\int f\,d\nu_k.
\]
Since both functions are nonnegative, this is \eqref{eq:mass}.
\end{proof}
Next lemma gives additional regularity for the minimizer $u$ and identifies
the fractional Dunkl operator acting on $u$ in the strong sense. 
\begin{lemma}\label{lem:regularity}
We have
\[
  f-\mu\in L^1(\nu_k)\cap L^2(\nu_k)
\]
and
\begin{equation}\label{eq:strong-op}
  A_{k}^{s/2}u=f-\mu \text{ almost everywhere}.
\end{equation}
  Consequently $u\in H_k^s$ and
\begin{equation}\label{eq:E-strong}
  \mathcal{E}_{k,s}(u,\psi)
   =\int_{\mathbb{R}^d}(f-\mu)\psi\,d\nu_k
\end{equation}
for every $\psi\in X_{k,s}$.  If $s\ge1$, then $u\in H_k^1$.
\end{lemma}

\begin{proof}
Since $0\le\mu\le\lambda$ and $\mu\in L^1$, we have
$\mu\in L^2$ with
\[
  \left\lVert \mu\right\rVert_2^2\le\lambda\left\lVert \mu\right\rVert_1.
\]
Thus $f-\mu\in L^1\cap L^2$.
Equation \eqref{eq:mu-test} says, distributionally,
\[
  A_{k}^{s/2}u=f-\mu.
\]
Taking the Dunkl transform gives
\[
  \left\lvert \xi\right\rvert^{s}\mathcal{F}_{k}u(\xi)=\mathcal{F}_{k}(f-\mu)(\xi)
\]
almost everywhere.  The right side is in $L^2$, proving \eqref{eq:strong-op} and
$u\in H_k^s$.  Identity \eqref{eq:E-strong} follows by Plancherel.
If $s\ge1$, then
$1+\left\lvert \xi\right\rvert^2\le 2\bigl(1+\left\lvert \xi\right\rvert^{2s}\bigr)$, so $H_k^s\subset H_k^1$.
\end{proof}
In the next lemma, we show that the reaction term $\eta$ is active only where the minimizer $u$ vanishes, more precisely:
\begin{lemma}\label{lem:complementarity}
We have
\begin{equation}\label{eq:eta-u-zero}
  \eta u=0\qquad\text{almost everywhere}.
\end{equation}
Consequently,
\begin{equation}\label{eq:mu-lambda-omega}
  \mu=\lambda\qquad\text{almost everywhere on }\Omega:=\{u>0\}.
\end{equation}
\end{lemma}

\begin{proof}
Taking $\psi=u$ in \eqref{eq:E-strong} gives
\[
  \mathcal{E}_{k,s}(u)=\int(f-\mu)u\,d\nu_k.
\]
On the other hand, \eqref{eq:energy-identity} gives
\[
  \mathcal{E}_{k,s}(u)=\int(f-\lambda)u\,d\nu_k.
\]
Subtracting yields
\[
  \int(\lambda-\mu)u\,d\nu_k=\int\eta u\,d\nu_k=0.
\]
Both factors are nonnegative, hence \eqref{eq:eta-u-zero} follows.  Since $u>0$ on $\Omega$,
we obtain $\eta=0$, equivalently $\mu=\lambda$, almost everywhere on $\Omega$.
\end{proof}

\section{Proof of the main theorems}\label{sec: man thm proof}

\begin{proof}[Proof of Theorem~\ref{thm:decomp}]
Let $u$ be the minimizer from Theorem~\ref{thm:minimizer}, let $\eta$ be defined by
\eqref{eq:eta-def}, and let $\mu=\lambda-\eta$.
The Lewy--Stampacchia estimate \eqref{eq:eta-density} gives $0\le\mu\le\lambda$.
Lemma~\ref{lem:mass} gives mass conservation, Lemma~\ref{lem:regularity} gives the
strong identity \eqref{eq:decomp} and $u\in H_k^s$, and
Lemma~\ref{lem:complementarity} gives $\mu=\lambda$ on $\Omega=\{u>0\}$.
Therefore
\[
  \lambda\nu_k(\Omega)
   =\int_\Omega\mu\,d\nu_k
   \le\left\lVert \mu\right\rVert_{L^1(\nu_k)}
   =\left\lVert f\right\rVert_{L^1(\nu_k)}.
\]

If $s\ge1$, Lemma~\ref{lem:regularity} gives $u\in H_k^1$.
Since $u\ge0$, the set $\{u\ne0\}$ agrees with $\Omega$ up to null sets, and
Lemma~\ref{lem:reflection-localization} yields
\eqref{eq:gradk-support} and \eqref{eq:omega-sharp-measure}.

Finally suppose $f$ is $G$-invariant.  The form $\mathcal{E}_{k,s}$, the measure $\nu_k$,
the cone $\mathcal{K}_{k,s}$, and the linear functional
$v\mapsto\int(f-\lambda)v\,d\nu_k$ are all $G$-invariant.  Hence for every $g\in G$,
$u\circ g$ is also a minimizer.  Uniqueness implies $u\circ g=u$.  It follows from
\eqref{eq:strong-op} that $\mu$ is $G$-invariant as well.  Thus $\Omega$ is
$G$-invariant and $\Omega^\sharp=\Omega$.
\end{proof}

We now take $s=1$ in Theorem~\ref{thm:decomp} and prove Theorem~\ref{thm:weak}.

\begin{proof}[Proof of Theorem~\ref{thm:weak}]
We first prove the estimate for
$f\in L^1(\nu_k)\cap L^2(\nu_k)$.  Write $f=f^+-f^-$ and fix $\lambda>0$.
Apply Theorem~\ref{thm:decomp} to $f^+$ and $f^-$ at level $\lambda$:
\[
  f^\pm=\mu^\pm+A_{k}^{1/2}u^\pm,
\]
where
\[
  0\le\mu^\pm\le\lambda,
  \qquad
  \left\lVert \mu^\pm\right\rVert_1=\left\lVert f^\pm\right\rVert_1,
\]
and, with $\Omega_\pm=\{u^\pm>0\}$,
\[
  \lambda\nu_k(\Omega_\pm)\le\left\lVert f^\pm\right\rVert_1.
\]
Set
\[
  \mu=\mu^+-\mu^- ,\qquad
  u=u^+-u^- ,\qquad
  \Omega=\Omega_+\cup\Omega_-.
\]
Then
\begin{equation}\label{eq:signed-decomp}
  f=\mu+A_{k}^{1/2}u.
\end{equation}
Applying $\mathcal{R}_{k}=\nabla_{k}A_{k}^{-1/2}$ in $L^2$ gives
\begin{equation}\label{eq:R-decomp}
  \mathcal{R}_{k} f=\mathcal{R}_{k}\mu+\nabla_{k} u.
\end{equation}

Define
\[
  \Omega^\sharp
  =\Omega\cup
   \bigcup_{\substack{\alpha\in R_+\\k(\alpha)>0}}\sigma_\alpha(\Omega).
\]
Lemma~\ref{lem:reflection-localization}, applied to $u^+$ and $u^-$, gives
$\nabla_{k} u=0$ almost everywhere off $\Omega^\sharp$.  Hence, up to a null set,
\begin{equation}\label{eq:level-inclusion}
  \{\left\lvert \mathcal{R}_{k} f\right\rvert_{\ell^2}>\lambda\}
  \subset
  \Omega^\sharp\cup\{\left\lvert \mathcal{R}_{k}\mu\right\rvert_{\ell^2}>\lambda\}.
\end{equation}
Since $\nu_k$ is reflection invariant,
\begin{align}
 \lambda\nu_k(\Omega^\sharp)
 &\le(M_k+1)\lambda\nu_k(\Omega)\notag\\
 &\le(M_k+1)
   \bigl(\left\lVert f^+\right\rVert_1+\left\lVert f^-\right\rVert_1\bigr)
  =(M_k+1)\left\lVert f\right\rVert_1.
 \label{eq:omega-final}
\end{align}
Also $\left\lvert \mu\right\rvert\le\lambda$ and
\[
  \int\left\lvert \mu\right\rvert\,d\nu_k
  \le\int(\mu^++\mu^-)\,d\nu_k
  =\left\lVert f\right\rVert_1,
\]
so
\begin{equation}\label{eq:mu-L2}
  \left\lVert \mu\right\rVert_2^2\le\lambda\left\lVert f\right\rVert_1.
\end{equation}
By Chebyshev and the $L^2$ isometry \eqref{eq:L2-isometry},
\begin{align}
 \lambda\nu_k\bigl(\{\left\lvert \mathcal{R}_{k}\mu\right\rvert_{\ell^2}>\lambda\}\bigr)
 &\le\lambda^{-1}\left\lVert \mathcal{R}_{k}\mu\right\rVert_2^2\\
 &=\lambda^{-1}\left\lVert \mu\right\rVert_2^2
 \le\left\lVert f\right\rVert_1.
 \label{eq:Rmu-final}
\end{align}
Combining \eqref{eq:level-inclusion}, \eqref{eq:omega-final}, and
\eqref{eq:Rmu-final} gives
\[
  \lambda\nu_k\bigl(\{\left\lvert \mathcal{R}_{k} f\right\rvert_{\ell^2}>\lambda\}\bigr)
  \le(M_k+2)\left\lVert f\right\rVert_1.
\]
Taking the supremum over $\lambda>0$ proves \eqref{eq:weak-main} on
$L^1\cap L^2$.  The standard weak-type extension principle for an operator initially
defined on $L^2$ then gives a unique extension to all of $L^1(\nu_k)$ with the same
constant.

If $f$ is $G$-invariant, then $f^\pm$, $u^\pm$, and $\Omega_\pm$ are $G$-invariant.
Thus $\Omega^\sharp=\Omega$ and the first term in \eqref{eq:omega-final} is bounded by
$\left\lVert f\right\rVert_1$ instead of $(M_k+1)\left\lVert f\right\rVert_1$.  This yields \eqref{eq:weak-inv}.
\end{proof}

\section{Application to Dunkl--Schr\"odinger Riesz transforms}\label{sec: application}

Let $V\geq 0$ be locally integrable with respect to $\nu_k$, and consider
the Dunkl--Schr\"odinger operator
\[
   \mathcal{L}_{k,V}:=-\Delta_k+V.
\]
We define the associated vector of the Dunkl--Schr\"odinger Riesz transforms by
\[
   \mathcal R_{k,V}f
   :=
   \bigl(
      T_1\mathcal{L}_{k,V}^{-1/2}f,\ldots,
      T_d\mathcal{L}_{k,V}^{-1/2}f
   \bigr).
\]
The semigroup perturbation argument used for Schr\"odinger operators in \cite{Dziubanski2026} gives
\begin{equation}\label{eq:dunkl-comparison}
   \bigl\|
      (-\Delta_k)^{1/2}\mathcal{L}_{k,V}^{-1/2}f
   \bigr\|_{L^1(\nu_k)}
   \leq 2\|f\|_{L^1(\nu_k)}.
\end{equation}
Indeed, the proof uses only the positivity and conservativity of the Dunkl
heat semigroup, its domination over the Schr\"odinger semigroup, and the
Duhamel perturbation formula. Thus, the argument of
\cite[Lemma~3.1]{Dziubanski2026} carries over to the Dunkl setting; see also
\cite[Remark~3.12]{Dziubanski2026}.

As a consequence of \eqref{eq:weak-main}, we obtain the following
vector-valued endpoint estimate.

\begin{corollary}\label{cor:dunkl-schrodinger}
For every real-valued $f\in L^1(\nu_k)$,
\[
   \|\mathcal R_{k,V}f\|_{L^{1,\infty}(\nu_k;\ell^2)}
   \leq 2(M_k+2)\|f\|_{L^1(\nu_k)}.
\]
In particular, the constant is independent of the nonnegative potential
$V$; its only dependence on the Dunkl structure is that already present
in $M_k$.
\end{corollary}

\begin{proof}
Set
\[
   A_{k,V}:=(-\Delta_k)^{1/2}\mathcal{L}_{k,V}^{-1/2}.
\]
The functional-calculus factorization
\[
   T_j\mathcal{L}_{k,V}^{-1/2}
   =
   T_j(-\Delta_k)^{-1/2}A_{k,V},
   \qquad j=1,\ldots,d,
\]
gives
\[
   \mathcal R_{k,V}f=\mathcal R_k(A_{k,V}f).
\]
Applying \eqref{eq:weak-main} followed by
\eqref{eq:dunkl-comparison}, we obtain
\[
\begin{aligned}
   \|\mathcal R_{k,V}f\|_{L^{1,\infty}(\nu_k;\ell^2)}
   &\leq (M_k+2)\|A_{k,V}f\|_{L^1(\nu_k)}  \\
   &\leq 2(M_k+2)\|f\|_{L^1(\nu_k)}.
\end{aligned}
\]
The factorization is initially valid on
$L^1(\nu_k)\cap L^2(\nu_k)$; the general case follows by density and the
weak-type estimate.
\end{proof}

\begin{remark}
To the best of our knowledge, the dimension and potential-free full-vector weak-type $(1,1)$ estimate is new even for classical Schr\"odinger operators $L=-\Delta+V$ with arbitrary nonnegative potential locally integrable potential $V$. 
\end{remark}


\subsection*{Declaration of generative AI and AI-assisted technologies in the manuscript preparation process}

During the preparation of this manuscript, the author used OpenAI's ChatGPT solely for language editing and to improve the readability and clarity of the mathematical exposition. All mathematical statements and proofs were independently verified by the author. The author reviewed and revised all AI-assisted content as necessary and takes full responsibility for the final version of the article.

\subsection*{Acknowledgments} 
The author is supported by Institute Postdoctoral Fellowship from IIT Bombay.

\subsection*{Data availability}
Data availability is not applicable.

\subsection*{Competing interests}
The author declares that he has no competing interests.

\end{document}